\documentclass[preprint,12pt]{elsarticle}

\usepackage{graphicx}

\usepackage{amssymb}
\usepackage{amsthm}

\usepackage{url}
\newtheorem{theorem}{Theorem}[section]

\newtheorem{remark}{Remark}[section]

\newtheorem{definition}{Definition}[section]
\newtheorem{example}{Example}[section]

\begin{document}
\begin{frontmatter}


\title{The equivalence between doubly nonnegative relaxation and semidefinite relaxation
for binary quadratic programming problems\tnoteref{s1}}
 \tnotetext[s1]{This work is supported by National
Natural Science Foundation of China (Grant No. 11071158)}

\author[shu]{Chuan-Hao Guo\corref{cor1}}
\ead{guo-ch@live.cn}

\author[shu]{Yan-Qin Bai\corref{cor2}}
\ead{yqbai@shu.edu.cn}

\author[shu]{Li-Ping Tang\corref{cor2}}
\ead{tanglipinggs@163.com}


\address[shu]{Department of Mathematics, Shanghai University,
Shanghai 200444, China}

\begin{abstract}
It has recently been shown (Burer, Math. Program Ser. A 120:479-495,
2009) that a large class of NP-hard nonconvex quadratic programming
problems can be modeled as so called completely positive programming
problems, which are convex but still NP-hard in general. A basic
tractable relaxation is gotten by doubly nonnegative relaxation,
resulting in a doubly nonnegative programming. In this paper, we
prove that doubly nonnegative relaxation for binary quadratic
programming (BQP) problem is equivalent to a tighter semidifinite
relaxation for it. When problem (BQP) reduces to max-cut (MC)
problem, doubly nonnegative relaxation for it is equivalent to the
standard semidifinite relaxation. Furthermore, some compared
numerical results are reported.
\end{abstract}

\end{frontmatter}
\section{Introduction}\label{sec: Introduction}
In this paper, we consider the following binary quadratic
programming problem
$$\rm{(BQP)}\ \ \ \ \
\begin{array}{lll}
&\min &x^TQx+2c^Tx\\
&{\rm s.t.} &a_i^Tx=b_i,\ i=1,2,\ldots,m,\\
&&x\in\{-1,1\}^n,\\
 \end{array}
 $$
  where $x\in R^n$ is the variable, $Q\in R^{n\times n}$, $c\in R^n$, $a_i\in R^n$ and $b_i\in R$ for all $i\in
  I:=\{1,2,\dots,m\}$ are the data. Without loss of generality, $Q$ is symmetric,
  and we assume $Q$ is not positive semidefinite, which implies
  generally that problem (BQP) is nonconvex and NP-hard
  \cite{pv1991}.

  Problem (BQP) arises in many applications, such as financial
analysis \cite{my1980}, molecular conformation problem \cite{pr1994}
and cellular radio channel assignment \cite{cs1995}. Many
combinatorial optimization problems are special cases of problem
(BQP), such as max-cut problem \cite{gw1995}. For solving this type
of problem, a systematic survey of the solution methods can be found
in Chapter 10 in \cite{ls2006} and the references therein.

It is well-known that semidefinite relaxation (SDR) is a powerful,
computationally efficient approximation technique for a host of very
difficult optimization problems, for instance, max-cut problem
\cite{gw1995}, Boolean quadratic program \cite{prw1995}. It also has
been at the center of some of the very exciting developments in the
area of signal processing and communications
\cite{mdwlc2002,msjcc2009}. The standard SDR for problem (BQP) is as
follows:
$$\rm{(SDR)}\ \ \ \ \
\begin{array}{lll}
&\min &X\bullet Q+2c^Tx\\
&{\rm s.t.} &a_i^Tx=b_i,\ \forall i\in I,\\
&&a_i^TXa_i=b_i^2,\ \forall i\in I,\\
&&X_{ii}=1,\ \forall i=1,2,\ldots,n,\\
&&X\succeq 0,
 \end{array}
 $$ where symbol $\bullet$ denotes the trace for any two conformal
 matrices. It is obviously that problem (SDR) is convex and gives a lower bound
 for problem (BQP) if the feasible set of problem (BQP) is nonempty.
 Moreover, if the optimal solution $(x^*,\ X^*)$ for problem (SDR) satisfy
 $X^*=x^*(x^*)^T$, then we can conclude that $x^*$ is an optimal
 solution for problem (BQP).

Recently, Burer \cite{b2009} proves that a large class of NP-hard
nonconvex quadratic programs with a mix of binary and continuous
variables can be modeled as so called completely positive programs,
which are convex but still NP-hard in general. In order to solve
such convex programs efficiently, a computable relaxed problem is
obtained by approximation the completely positive matrices with
doubly nonnegative matrices, resulting in a doubly nonnegative
programming \cite{b2010}, which can be efficiently solved by some
popular packages. For more details and developments of this
technique, one may refer to \cite{b2009,b2010,b2012,gy2010} and the
references therein.

In this paper, a tighter SDR problem and a doubly nonnegative
relaxation (DNNR) problem for problem (BQP) are established,
respectively, according to the features of the constraints in
problem (BQP) and the techniques of DNNP. And, we prove that doubly
nonnegative relaxation for problem (BQP) is equivalent to the
tighter semidefinite relaxation for it. Applying this result to
max-cut (MC) problem, it is shown that doubly nonnegative relaxation
for problem (MC) is equivalent to the standard semidefinite
relaxation for it. Moreover, some compared numerical results are
reported to illustrate the features of doubly nonnegative relaxation
and semidefinite relaxation, respectively.

The paper is organized as follows. In Section \ref{sec: New
relaxation for problem (BQP)}, a new tighter semidefinite relaxation
for problem (BQP) is proposed in Section \ref{sec: New tighter SDR
for problem (BQP)}. Problem (BQP) is relaxed to a doubly nonnegative
programming problem in Section \ref{sec:Doubly nonnegative
relaxation for problem (BQP)}. Section \ref{sec: Relationship
between relaxation problems} and Section \ref{sec: application to
max-cut problem} are devoted to show the equivalence of two
relaxation problems for problem (BQP) and problem (MC),
respectively. 
Some conclusions are given in Section \ref{sec:
conclusions}.


\section{New relaxation for problem (BQP)}\label{sec: New relaxation for problem (BQP)}
\subsection{New tighter SDR for problem (BQP)}\label{sec: New tighter SDR for problem
(BQP)}
First, note that problem (BQP) also can be relaxed to the following
problem by SDR
$$\widetilde{\rm{(SDR)}}\begin{array}{lll}
&\min &X\bullet Q+2c^Tx\\
&{\rm s.t.} &a_i^Tx=b_i,\ \forall i\in I,\\
&&a_i^TXa_i=b_i^2,\ \forall i\in I,\\
&&X_{ii}=1,\ \forall i=1,2,\ldots,n,\\
&&X-xx^T\succeq 0.
 \end{array}$$
If the optimal solution $(x^*,\ X^*)$ for problem
$\widetilde{\rm{(SDR)}}$ satisfy
 $X^*=x^*(x^*)^T$, it holds that $x^*$ also is an optimal
 solution for problem (BQP). On one hand, it is worth noting that
 \begin{equation}\label{e: conic relaxation}
X-xx^T\succeq 0\Longrightarrow X\succeq 0
\end{equation}
holds always, which further implies that any feasible solution of
problem $\widetilde{\rm{(SDR)}}$ is also feasible for problem (SDR).
It follows that $\rm{Opt(SDR)}\leq\rm{Opt(\widetilde{\rm{SDR)}}}$
since the two problems have the same objective functions, where
$\rm{Opt(\ast)}$ denotes the optimal value for problem $(\ast)$.
Therefore, we can conclude that problem $\widetilde{\rm{(SDR)}}$ is
a tighter SDR problem for problem (BQP) than problem (SDR).

On the other hand, we can easily verify that the constraint
$X-xx^T\succeq 0$ is nonconvex, since the quadratic term $-xx^T$ is
nonconvex. Thus, problem $\widetilde{\rm{(SDR)}}$ is nonconvex and
not solved by some popular packages for solving convex programs. In
order to establish the convex representation for problem
$\widetilde{\rm{(SDR)}}$, a crucial theorem is given below and the
details of its proof can be seen in Appendix A.5.5 Schur complement
in \cite{bv2004}.

\begin{theorem}\label{theo: Schur complement}
Let matrix $M\in S^n$ is partitioned as
$$M=\left[\begin{array}{ll}A & B\\ B^T & C\end{array}\right].$$ If
$\rm{det}A\neq 0$, the matrix $H=C-B^TA^{-1}B$ is called the Schur
complement of $A$ in $M$. Then,
we have the following relations:\\
(i) $M\succ 0$ if and only if $A\succ 0$ and $H\succ 0$.\\
(ii) If $A\succ 0$, then $M\succeq 0$ if and only if $H\succeq 0$.
\end{theorem}

According to Theorem \ref{theo: Schur complement}(ii) and (\ref{e:
conic relaxation}), it holds immediately that
\begin{equation}\label{e: conic equivalence}
\left[\begin{array}{ll} 1 & x^T \\ x & X
\end{array}\right]\succeq 0\Longleftrightarrow X-xx^T\succeq
0\Longrightarrow X\succeq 0,
\end{equation}
i.e., the constraint $X-xx^T\succeq 0$ can be equivalently
reformulated as $\left[\begin{array}{ll}1 & x^T\\ x &
X\end{array}\right]\succeq 0$, which is not only convex, but also
computable. So, problem $\widetilde{\rm{(SDR)}}$ is equivalently
reformulated as follows:
$${\rm{(SDR1)}}\begin{array}{lll}
&\min &X\bullet Q+2c^Tx\\
&{\rm s.t.} &a_i^Tx=b_i,\ \forall i\in I,\\
&&a_i^TXa_i=b_i^2,\ \forall i\in I,\\
&&X_{ii}=1,\ \forall i=1,2,\ldots,n,\\
&&\left[\begin{array}{ll}1 & x^T\\ x & X\end{array}\right]\succeq 0,
 \end{array}$$
which is not only convex in form, also can be efficiently solved by
some popular packages for solving convex programs.

Here, some examples are given to show that problem (SDR1) is a
tighter relaxation problem compared to problem (SDR), and the
corresponding numerical results further show that problem (SDR1) is
more efficient than problem (SDR). These examples are solved by
\texttt{CVX}, a package for specifying and solving convex programs
\cite{gb2011}.

\begin{example}\label{ex: 1_sdr_sdr1}This is a two dimensional nonconvex problem with one linear equality
constraint, the corresponding coefficients are selected as follows:
$$Q=\left[\begin{array}{cc}0 & -3 \\ -3 & -20\end{array}\right],\
c=\left[\begin{array}{cc}-8\\9\end{array}\right],\ \ A=[10,\ -10],\
\ b=0,$$ where $A=[a_1,a_2,\ldots,a_m]^T$,
$b=[b_1,b_2,\ldots,b_m]^T$.

On one hand, we use \texttt{CVX} to solve problem (SDR), then we
obtain $\rm{Opt(SDR)}=-\infty$, since problem (SDR) is unbounded
below. On the other hand, when problem (SDR1) is solved, it follows
that $\rm{Opt(SDR1)}=-28$ with $X=\left[\begin{array}{cc}1 & 1 \\ 1
& 1\end{array}\right]$ and $x=[-1,\ -1]^T$. Note that the
relationship $X=xx^T$ holds, thus we can conclude that $x=[-1,\
-1]^T$ also is an optimal solution for problem (BQP). The results
show that problem (SDR1) is more tighter and efficient than problem
(SDR) for this problem.
\end{example}

\begin{example}\label{ex: 2_sdr_sdr1}This problem is five dimensions with three linear equality
constraints, the corresponding coefficients are chosen as follows:
$$Q=\left[\begin{array}{cccccc}-52  &  31  &  49  &  -7  &   4\\
    31  & -16  & -50 &  -13  & -49\\
    49   &-50  &   8  &  44  & -30\\
    -7  & -13  &  44   & 36   & 12\\
     4  & -49  & -30   & 12  &  56\end{array}\right],\ \ \ \ \
c=\left[\begin{array}{cc}-20\\
    37\\
    43\\
    25\\
    -6\end{array}\right],$$
$$A=\left[\begin{array}{ccccc}4  &  10   & 29  &  14  &
    -36\\
    38  &   9   &  1  & -17  &  23\\
    48  &  39   &  5  & -17  & -13\end{array}\right],\ \ \ \ \
\ b=\left[\begin{array}{ccc}11\\
   -50\\
   -36\end{array}\right].$$
If this problem is solved by \texttt{CVX} with problem (SDR), then
it returns $\rm{Opt(SDR)}=-\infty$ since problem (SDR) is unbounded
below. When we use problem (SDR1) to solve this problem, we have
$\rm{Opt(SDR1)}=-307.548$, however, the relationship $X=xx^T$ is not
holds for this problem. Therefore, we obtain a tighter lower bound
$-307.548$ for original problem. These results also show that
problem (SDR1) is more effective than problem (SDR).
\end{example}


In fact, the constraint $x\in\{-1,+1\}^n$ in problem (BQP) further
imply that the following relationship
\begin{equation}\label{e: cut1}(1-x_i)(1-x_j)\geq 0\Rightarrow 1-x_i-x_j+X_{ij}\geq 0,\
\forall 1\leq i\leq j\leq n
\end{equation}
always hold. Combing with (\ref{e: cut1}) in problem (SDR1), we get
the following new semidefinite relaxation problem
$$\rm{(SDR2)}\begin{array}{lll}
&\min &X\bullet Q+2c^Tx\\
&{\rm s.t.} &a_i^Tx=b_i,\ \forall i\in I,\\
&&a_i^TXa_i=b_i^2,\ \forall i\in I,\\
&&X_{ii}=1,\ \forall i=1,2,\ldots,n,\\
&&1-x_i-x_j+X_{ij}\geq 0,\ \forall 1\leq i\leq j\leq n,\\
&&\left[\begin{array}{ll}1 & x^T\\ x & X\end{array}\right]\succeq 0.
 \end{array}$$

The above semidefinite relaxation problem $\rm{(SDR2)}$ is more
tighter than problem $\rm{(SDR1)}$ in form, since $\frac{n(n+1)}{2}$
inequality constraints are added into corresponding problem
$\rm{(SDR2)}$. Furthermore, we will prove that problem (SDR2) is
equivalent to another convex relaxation problem for problem (BQP) in
Section \ref{sec: Relationship between relaxation problems}.

Now, we test some problems to show that problem (SDR2) is tighter
than problem (SDR1) from the computational point of view. These
problems are of one of two types:

\begin{table}[htbp] \centering \footnotesize\caption{Statistics of the test problems}
\begin{tabular}{ lllllllllllllllllllllllllll}
 \hline\hline
  Type   &  &Instances  &  &$n$  & &$m$  & &Function \\
 \hline
 RdnBQP  &  &$50$       &  &$50$ & &$20$ & &\texttt{randn($\cdot$)}\\
 RdiBQP  &  &$50$       &  &$50$ & &$20$ & &\texttt{randi($[-10, 10], \cdot$)}\\
 \hline\hline
\end{tabular}
\label{tab: Statistics for sdr1_sdr2}
\end{table}

$\bullet$\ \ {\footnotesize\textbf{RdnBQP}}. We generate $50$
instances of problem (BQP) by MATLAB function
\texttt{randn($\cdot$)}. The symmetric matrix $Q$ is generated by
\texttt{tril(randn($\cdot$),\\-1)+triu(randn($\cdot$)',0)}, and all
instances are nonconvex.

$\bullet$\ \ {\footnotesize\textbf{RdiBQP}}. $50$ instances of
problem (BQP) are generated by MATLAB function \texttt{randi([-10,
10], $\cdot$)}. The symmetric matrix $Q$ is generated by
\texttt{randn($\cdot$)+randn($\cdot$)'}. Each element in the data
coefficients is a random integer number in the range $[-10, 10]$.
All instances are nonconvex binary quadratic programming problems.

To compare the performance of two relaxation problems for problem
(BQP), by using problem (SDR1) and problem (SDR2), respectively, we
use performance profiles as described in Dolan and Mor\'{e}'s paper
\cite{dm2002}. Our profiles are based on optimal values for problems
(SDR1) and (SDR2). These problems are solved by \texttt{CVX}, and
the results of performance are shown in Figure \ref{fig_sdr1_sdr2}.
From Figure \ref{fig_sdr1_sdr2}, it is obviously that the lower
bound which got from problem (SDR2) is much greater than that one of
from problem (SDR1), for test problems RdnBQP and RdiBQP,
respectively. Moreover, we find that optimal value of problem (SDR2)
is strictly greater than that of problem (SDR1) for test problems in
the experiment. Thus, the performance of problem (SDR2) is much
better than problem (SDR1) for solving problem (BQP) in some sense.
\begin{figure}[htp]
   \centering
    \includegraphics[width=67mm]{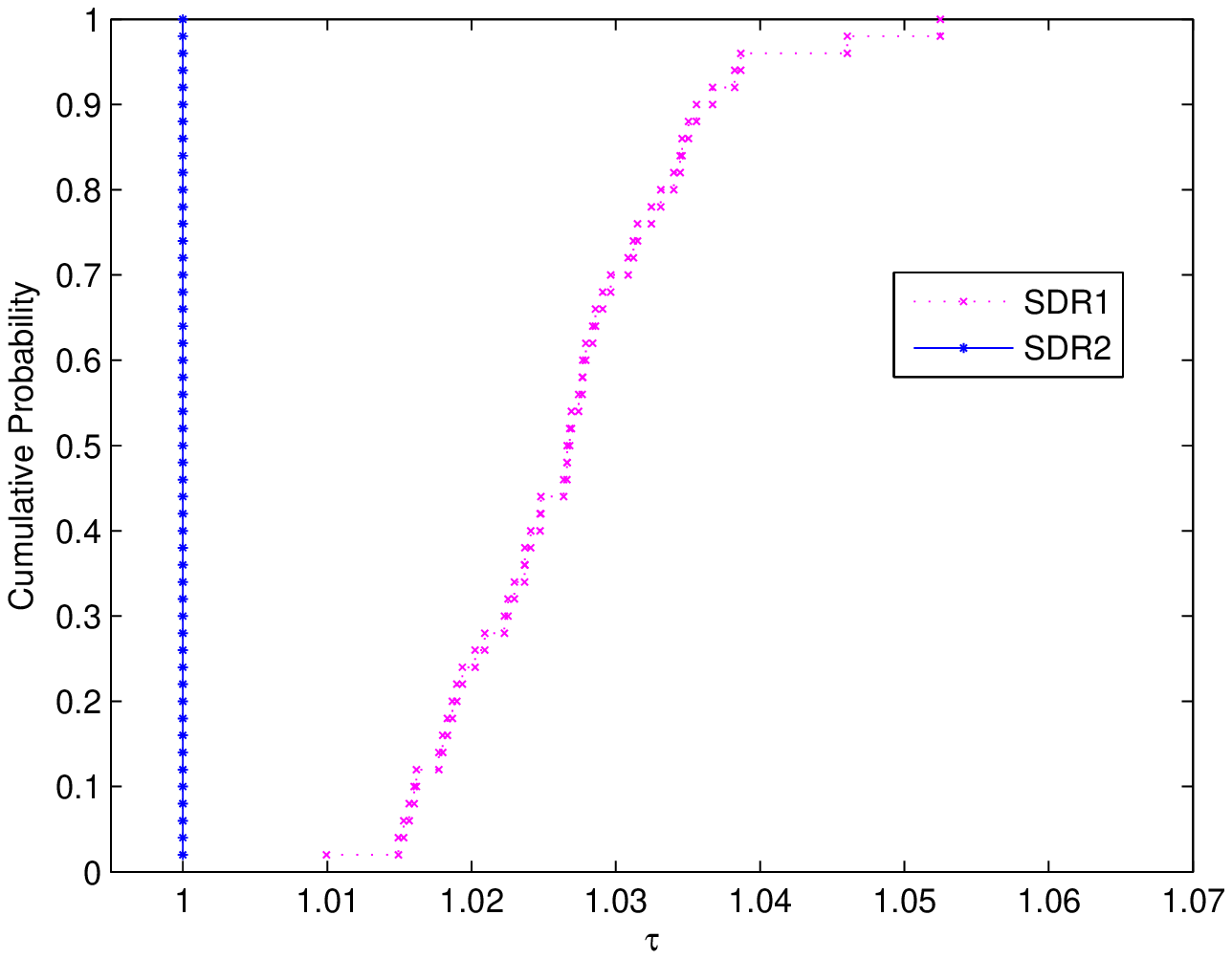}
    \includegraphics[width=67mm]{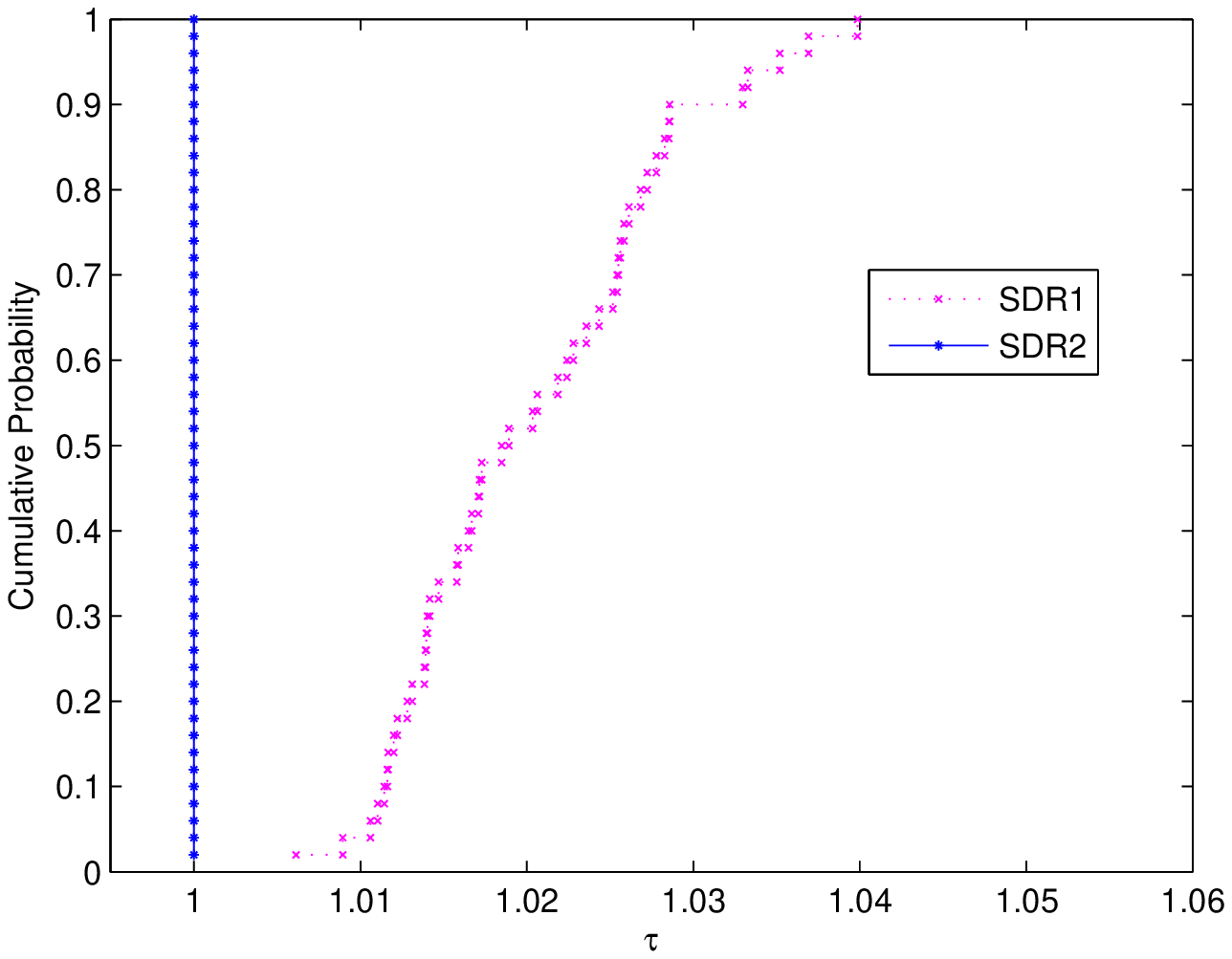}
      \caption{Left figure is based on optimal values of problems RdnBQP,
      right figure is based on optimal values of problems RdiBQP.}
     \label{fig_sdr1_sdr2}
 \end{figure}

\subsection{Doubly nonnegative relaxation for problem
(BQP)}\label{sec:Doubly nonnegative relaxation for problem (BQP)}
Recently, Burer \cite{b2009} has shown that a large class of NP-hard
nonconvex quadratic problems with binary constraints can be modeled
as so-called completely positive programs (CPP), i.e., the
minimization of a linear function over the convex cone of completely
positive matrices subject to linear constraints. Motivated by the
ideas, we first establish the CPP representation for problem (BQP),
and then give its doubly nonnegative relaxation (DNNR) formulation.
Subsequently, some compared numerical results are presented in this
section.

Let $z=\frac{1}{2}(e-x)$ in problem (BQP), it follows that
$z\in\{0,1\}^n$, and then problem (BQP) can be equivalently
reformulated as follows:
$$\widehat{\rm{(BQP)}}\begin{array}{lll}
&\min &4z^TQz-4z^T(Qe+c)+e^TQe+2c^Te\\
&{\rm s.t.} &2a_i^Tz=a_i^Te-b_i,\ \forall i\in I,\\
&&z\in\{0,\ 1\}^n,
 \end{array}
$$
where $e$ denote the vector of ones with appropriate dimension.
According to Theorem 2.6 in \cite{b2009} and similar to the analysis
in it, problem $\widehat{\rm{(BQP)}}$ can be further equivalently
transformed into the following CPP problem
$${\rm{(CPP)}}\begin{array}{lll}
&\min &4Q\bullet Z-4z^T(Qe+c)+e^TQe+2c^Te\\
&{\rm s.t.} &2a_i^Tz=a_i^Te-b_i,\ \forall i\in I,\\
&&4a_i^TZa_i=(a_i^Te-b_i)^2,\ \forall i\in I,\\
&&Z_{ii}=z_i,\ \forall i=1,2,\ldots,n,\\
&&\left[\begin{array}{ll}1 & z^T\\ z & Z\end{array}\right]\in
C_{1+n},
 \end{array}
$$
 where $C_{1+n}$ is defined as follows:
$$C_{1+n}:=\left\{X\in S^{1+n}:\ X=\sum\limits_{k\in
K}z^k(z^k)^T\right\}\cup\{0\},$$ and for some finite $\{z^k\}_{k\in
K}\subset R^{1+n}_+\backslash\{0\}$.

In view of the definition of convex cone in \cite{bv2004}, $C_{1+n}$
is a closed convex cone, and is called the completely positive
matrices cone. Thus, problem (CPP) is a convex problem. However,
problem (CPP) is NP-hard, since checking whether or not a given
matrix belongs to $C_{1+n}$ is NP-hard, which has been shown by
Dickinson and Gijen in \cite{dg201}. Thus, it has to be replaced by
some computable cones, which can efficiently approximate cone
$C_{1+n}$. Note that the convex cone $(S_n)^+$ is self-dual, and so
is the convex cone $S_n^+$, where $(S^n)^+$ and $S^n_+$ denotes the
cone of $n\times n$ nonnegative symmetric matrices and the cone of
$n\times n$ positive semidefinite matrices, respectively. Hence,
Diananda's decomposition theorem \cite{d1962} can be reformulated as
follows.
\begin{theorem}\label{the: Diananda's decomposition theorem}$C_{n}\subseteq S_{n}^+\cap
(S_{n})^+$ holds for all $n$. If $n\leq 4$, then $C_{n}=S_{n}^+\cap
(S_{n})^+$.
\end{theorem}
By the way, the matrices in $S_n^+\cap (S_n)^+$ sometimes are called
``doubly nonnegative". Of course, in dimension $n\geq 5$ there are
matrices which are doubly nonnegative but not completely positive,
the counterexample can be seen in \cite{bs2003}.

According to Theorem \ref{the: Diananda's decomposition theorem},
problem (CPP) can be relaxed to the following DNNP problem
$${\rm{(DNNP)}}\begin{array}{lll}
&\min &4Q\bullet Z-4z^T(Qe+c)+e^TQe+2c^Te\\
&{\rm s.t.} &2a_i^Tz=a_i^Te-b_i,\ \forall i\in I,\\
&&4a_i^TZa_i=(a_i^Te-b_i)^2,\ \forall i\in I,\\
&&Z_{ii}=z_i,\ \forall i=1,2,\ldots,n,\\
&&\left[\begin{array}{ll}1 & z^T\\ z & Z\end{array}\right]\in
S_{1+n}^+\cap(S_{1+n})^+.
 \end{array}
$$

Up to now, the other convex relaxation problem for problem (BQP) is
established, i.e., problem (DNNP), which is computable by some
popular packages for solving convex programs, such as \texttt{CVX},
etc.

Note that problem (DNNP) has $n+\frac{n(n+1)}{2}$ equality
constraints more than standard semidefinite relaxation problem
(SDR1), and $n$ equality constraints more than problem (SDR2),
respectively. Thus, the lower bound which get from problem (DNNP) is
much greater than that one of by problem (SDR1) and problem (SDR2),
respectively.

In the following, two types of problems are tested to show the
performance of problem (SDR1), problem (SDR2) and problem (DNNP),
respectively. The statistics of the test problems are chosen as
follows:
\begin{table}[htbp] \centering \footnotesize\caption{Statistics of the test problems}
\begin{tabular}{ lllllllllllllllllllllllllll}
 \hline\hline
  Type   &  &Instances  &  &$n$  & &$m$  & &Function \\
 \hline
 RdBQP  &  &$50$       &  &$50$ & &$25$ & &\texttt{rand($\cdot$)}\\
 RdsBQP  &  &$50$       &  &$50$ & &$25$ & &\texttt{rands($\cdot$)}\\
 \hline\hline
\end{tabular}
\label{tab: Statistics for sdr1_sdr2_dnnp}
\end{table}

$\bullet$\ \ {\footnotesize\textbf{RdBQP}}. For this type of
problems, we generate $50$ instances of problem (BQP) by using
MATLAB function \texttt{rand($\cdot$)}. The symmetric matrix $Q$ is
generated by \texttt{rand($\cdot$)+rand($\cdot$)'}, and all problems
are nonconvex.

$\bullet$\ \ {\footnotesize\textbf{RdsBQP}}. The coefficients of
$50$ instances of problem (BQP) are generated by using MATLAB
function \texttt{rands($\cdot$)}, and the symmetric matrix $Q$ is
generated by \texttt{rands($\cdot$)+rands($\cdot$)'}. All instances
are nonconvex.

We use performance profiles \cite{dm2002} to compare the performance
of problem (SDR1), problem (SDR2) and problem (DNNP), for problem
(BQP), respectively. The corresponding results of performance are
shown in Figure \ref{fig_sdr1_sdr2_dnnp_optval}. The profiles for
Figure \ref{fig_sdr1_sdr2_dnnp_optval} are based on optimal values
of problem (SDR1), problem (SDR2) and problem (DNNP), respectively,
and these problems are solved by \texttt{CVX}. From Figure
\ref{fig_sdr1_sdr2_dnnp_optval}, it is obviously that the
performances of problem (SDR2) and problem (DNNP) are almost the
same, which are better than that one of problem (SDR1), for problems
RdBQP and RdsBQP, respectively. Thus, we can conclude that it is
more efficient to use problem (SDR2) and problem (DNNP) than problem
(SDR1) to solve problem (BQP), from the point of view of optimal
values.

Furthermore, we will show the equivalence of the problems (DNNP) and
(SDR2) in Section \ref{sec: Relationship between relaxation
problems}.

\begin{figure}[htp]
   \centering
    \includegraphics[width=67mm]{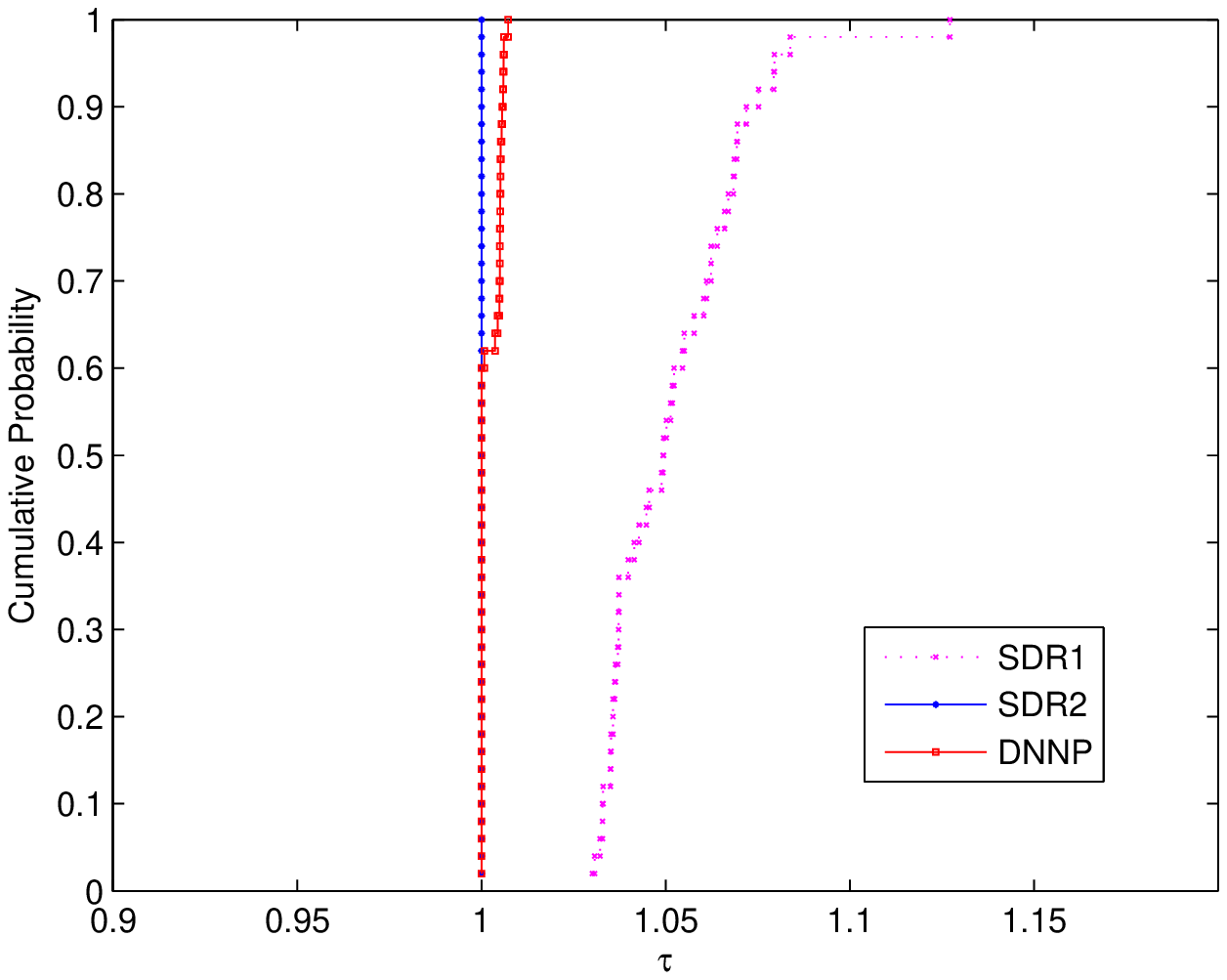}
    \includegraphics[width=67mm]{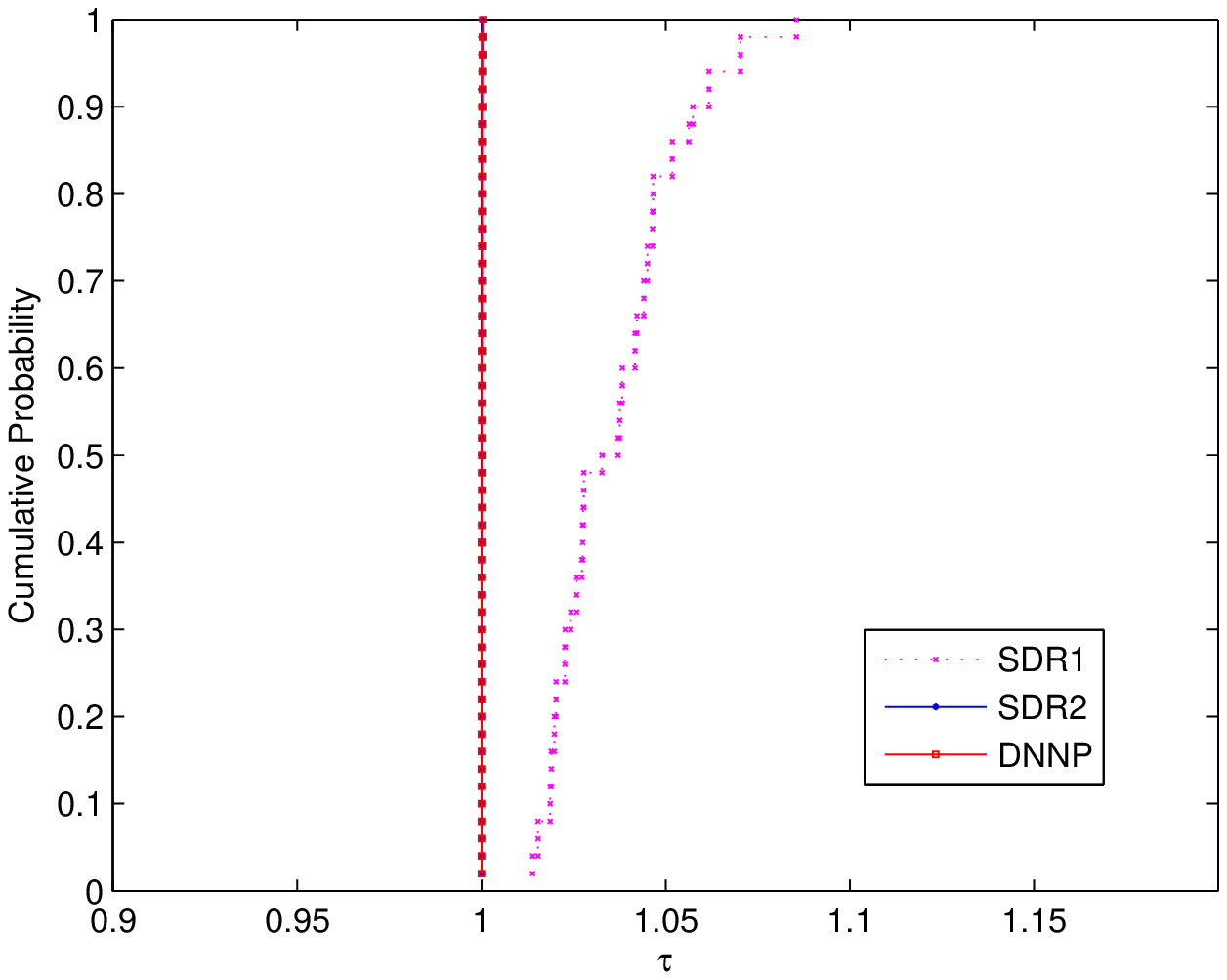}
      \caption{Left figure is based on optimal values of problems RdBQP,
      right figure is based on optimal values of problems RdsBQP.}
     \label{fig_sdr1_sdr2_dnnp_optval}
 \end{figure}

\section{Relationship between relaxation problems}\label{sec: Relationship between relaxation problems}
In this section, we will investigate the relationship between two
relaxation problems (SDR2) and (DNNP). First of all, the definition
of the equivalence of two optimization problems is defined as
follows.

\begin{definition}\label{def: equivalence} We call two problems are
equivalent if they satisfy the following two conditions:

(i)  If from a solution of one problem, a solution of the other
problem is readily found, and vice versa.

(ii) The two problems have the same optimal value.
\end{definition}

Now, based on the above Definition \ref{def: equivalence}, the main
theorem is given below.

\begin{theorem}\label{theo: equivalence problem (SDR2) and problem (DNNP)}
Suppose that the feasible sets $\rm{Feas(SDR2)}$ and
$\rm{Feas(DNNP)}$ are all nonempty. Then, two problems (SDR2) and
(DNNP) are equivalent.
\end{theorem}
\begin{proof}
The proof can be divided into two parts. First of all, we will prove
that $\rm{Opt(SDR2)\geq Opt(DNNP)}$.

Suppose that $(x^*,X^*)$ is an optimal solution of problem
$\rm(SDR2)$, let $Z_{ij}=\frac{1}{4}(1-x_i^*-x_j^*+X_{ij}^*)$ and
$z_i=\frac{1}{2}(1-x_i^*),\ \forall 1\leq i\leq j\leq n$, i.e.,
\begin{equation}\label{e: 2}Z=\frac{1}{4}(ee^T-e(x^*)^T-x^*e^T+X^*),\
z=\frac{1}{2}(e-x^*).\end{equation}

By $a_i^Tx^*=b_i$ for all $i\in I$ and (\ref{e: 2}), we have
\begin{equation}\label{e: 3}
a_i^Tx^*=a_i^T(e-2z)=b_i\Rightarrow 2a_i^Tz=a_i^Te-b_i,\ \forall
i\in I.
\end{equation}

From (\ref{e: 2}) and $a_i^TX^*a_i=b_i^2$ for all $i\in I$, it
follows that
\begin{equation}\label{e: 4}\begin{array}{ll}
a_i^TX^*a_i&=a_i^T(4Z-ee^T+e(x^*)^T+x^*e^T)a_i=b_i^2\\
&\Rightarrow 4a_i^TZa_i=(a_i^Te-b_i)^2,\ \forall i\in I.
\end{array}
\end{equation}

Again from (\ref{e: 2}), which imply that
\begin{equation}\label{e: 5}Z_{ii}=\frac{1}{4}(1-2x_i^*+X_{ii}^*)=\frac{1}{2}(1-x_i^*)=z_i,\ \forall i\in I,\end{equation}
since $X_{ii}^*=1$.

From $1-x_i^*-x_j^*+X_{ij}^*\geq 0, \forall 1\leq i\leq j\leq n$, it
holds that
\begin{equation}\label{e: 6}Z_{ij}\geq 0, \forall 1\leq
i\leq j\leq n,\end{equation} which combining with (\ref{e: 5}),
further imply that
\begin{equation}\label{e: 7}
z_i\geq 0, \forall 1\leq i\leq n.
\end{equation}

By Theorem \ref{theo: Schur complement}(ii) and (\ref{e: 2}), it
follows that
\begin{equation}\label{e: 8}\begin{array}{ll}
Z-zz^T&=\frac{1}{4}(ee^T-e(x^*)^T-x^*e^T+X^*)-\frac{1}{4}(e-x^*)(e-x^*)^T\\
&=\frac{1}{4}(X^*-x^*(x^*)^T)\succeq 0.
\end{array}
\end{equation}
Combining (\ref{e: 8}) with (\ref{e: 3}), (\ref{e: 4}), (\ref{e: 6})
and (\ref{e: 7}), it follows that $(z,Z)$ defined by (\ref{e: 2}) is
a feasible solution for problem $\rm{(DNP)}$.

Moreover, again from (\ref{e: 2}), we have
$$\begin{array}{ll}
4Q\bullet
Z-4z^T(Qe+c)+e^TQe+2c^Te\\
=Q\bullet(ee^T-e(x^*)^T-x^*e^T+X^*)-2(e-x^*)^T(Qe+c)+e^TQe+2c^Te\\
=Q\bullet X^*+2c^Tx^*=\rm{Opt(SDR2)},
\end{array}
$$
which further imply that $\rm{Opt}(DNNP)\leq\rm{Opt}(SDR2)$.

On the other hand, given an optimal solution $(z^*,Z^*)$ to problem
(DNNP), and let
\begin{equation}\label{e: 9}
X_{ij}=1-2z_i^*-2z_j^*+4Z_{ij}^*,\ x_i=1-2z^*_i,\ \forall 1\leq
i\leq j\leq n,
\end{equation}
which imply that
\begin{equation}\label{e: 10}
X_{ii}=1-4z^*_i+4Z^*_{ii}=1
\end{equation}
since $Z^*_{ii}=z^*_i,\ \forall i=1,2,\ldots,n$. Moreover,
\begin{equation}\label{e: 11}\begin{array}{ll}
1-x_i-x_j+X_{ij}&=1-(1-2z_i^*)-(1-2z_j^*)+1-2z_i^*-2z_j^*+4Z_{ij}^*\\
&=4Z_{ij}^*\geq 0,\ \forall 1\leq i\leq j\leq n.
\end{array}\end{equation}

From (\ref{e: 9}) and $2a_i^Tz^*=a_i^Te-b_i,\ \forall i\in I$, it
follows that
\begin{equation}\label{e: 12}
a_i^Tx=a_i^T(e-2z^*)=b_i, \ \forall i\in I.
\end{equation}

Again from (\ref{e: 9}) and $4a_i^TZ^*a_i=(a_i^Te-b_i)^2,\ \forall
i\in I$, we have
\begin{equation}\label{e: 13}\begin{array}{ll}
a_i^TXa_i&=a_i^T(ee^T-2e(z^*)^T-2z^*e^T+4Z^*)a_i\\
&=b_i^2,\ \forall i\in I.
\end{array}
\end{equation}

From (\ref{e: 9}) and Theorem \ref{theo: Schur complement}(ii), it
holds that
\begin{equation}\label{e: 14}\begin{array}{lll}
X-xx^T&=ee^T-2z^*e^T-2e(z^*)^T+4Z^*-(e-2z^*)(e-2z^*)^T\\
&=4(Z^*-z^*(z^*)^T)\succeq 0.
\end{array}\end{equation}

By (\ref{e: 11}), (\ref{e: 12}), (\ref{e: 13}) and (\ref{e: 14}), we
can conclude that $(x,X)$ defined by (\ref{e: 9}) is a feasible
solution for problem $\rm{(SDR2)}$. Furthermore, we have
$$\begin{array}{ll}X\bullet Q+2c^Tx&=(ee^T-2e(z^*)^T-2z^*e^T+4Z^*)\bullet Q+2c^T(e-2z^*)\\
&=4Z^*\bullet Q-4(z^*)^T(Qe+c)+e^TQe+2c^Te\\&=\rm{Opt(DNNP)},
\end{array}$$
which imply that $\rm{Opt}(SDR2)\leq\rm{Opt}(DNNP)$. Summarizing the
analysis above and according to Definition \ref{def: equivalence},
we can conclude that problem (DNNP) is equivalent to problem (SDR2).
\end{proof}

Although $\rm{Opt}(SDR2)=\rm{Opt}(DNNP)$ in view of Theorem
\ref{theo: equivalence problem (SDR2) and problem (DNNP)} and
Definition \ref{def: equivalence}, problem (DNNP) has $n$ equality
constraints more than problem (SDR2) in form. So, the amount of
computation for solving problem (DNNP) may be much greater than that
one of solving problem (SDR2). In order to illustrate this point of
view, the compared performance results are shown in Figure
\ref{fig_sdr1_sdr2_dnnp_itenumb} and Figure
\ref{fig_sdr1_sdr2_dnnp_cpu}, respectively, which are based on the
number of iterations and CPU time for solving problems RdBQP and
RdsBQP. The results in Figure \ref{fig_sdr1_sdr2_dnnp_itenumb} show
that the performance of problem (SDR2) is better than that one of
problem (DNNP) for problems RdBQP, but the performance of problem
(DNNP) is better than that one of problem (SDR2) for problems
RdsBQP, in view of the points of the number of iterations. From the
results of the performance of CPU time, it is obviously that problem
(SDR2) is more efficient than problem (DNNP) for solving problems
RdBQP and RdsBQP, respectively. Summarizing the analysis above, we
can efficiently solving problem (BQP) by soling problem (SDR2) or
problem (DNNP) in practice.
\begin{figure}[htp]
   \centering
    \includegraphics[width=67mm]{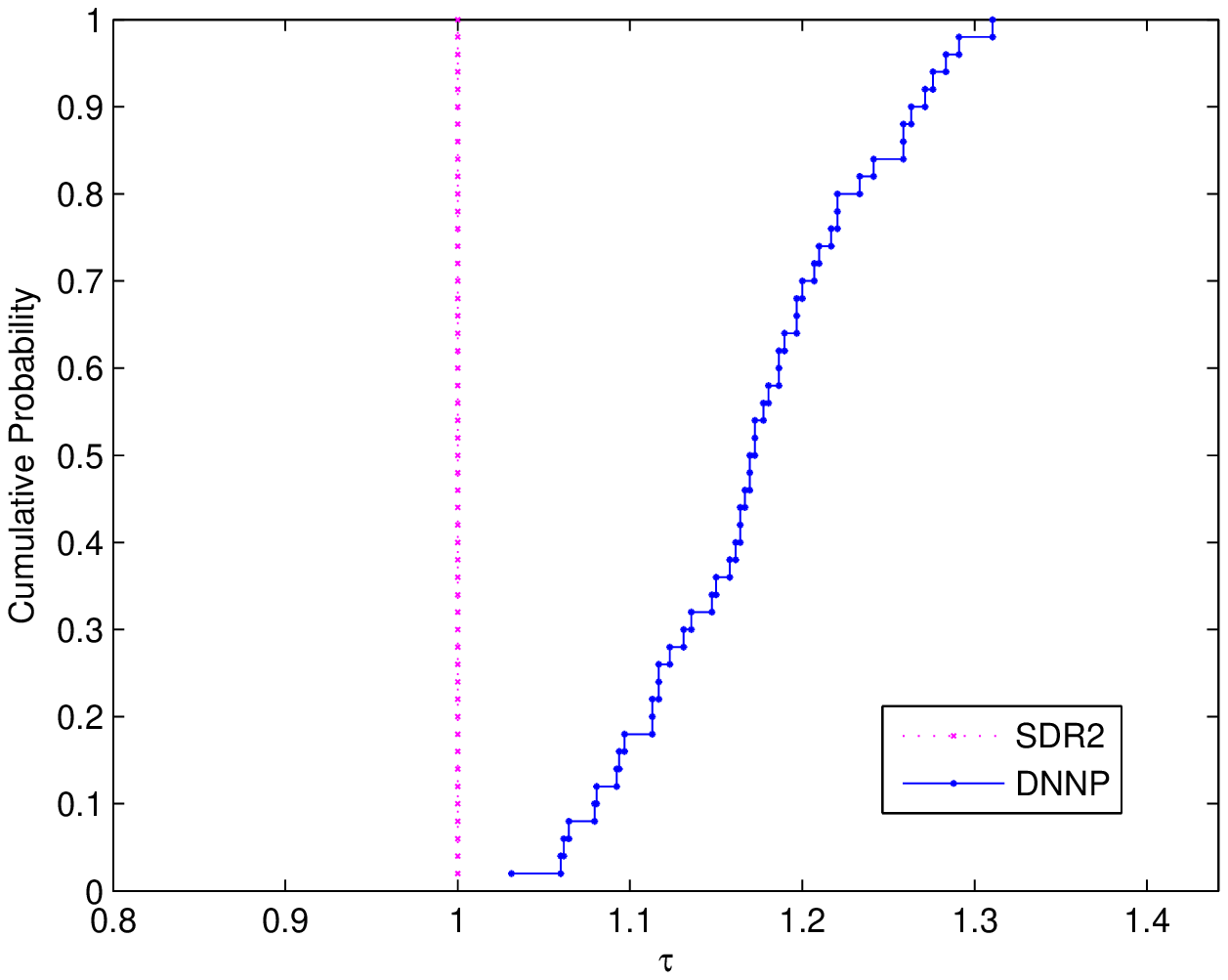}
    \includegraphics[width=67mm]{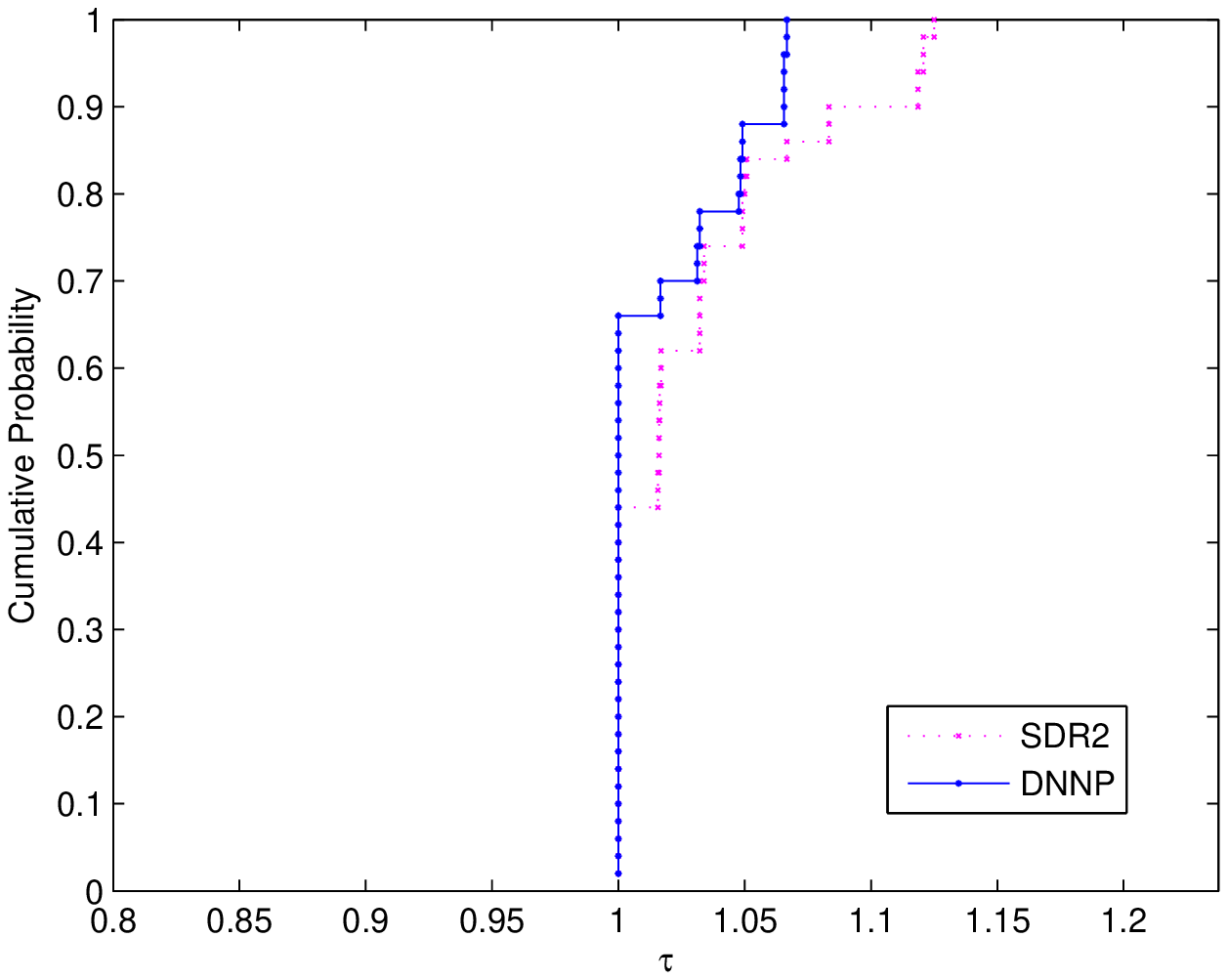}
      \caption{Left figure is based on the number of iterations of problems RdBQP,
      right figure is based on the number of iterations of problems RdsBQP.}
     \label{fig_sdr1_sdr2_dnnp_itenumb}
 \end{figure}
\begin{figure}[htp]
   \centering
    \includegraphics[width=67mm]{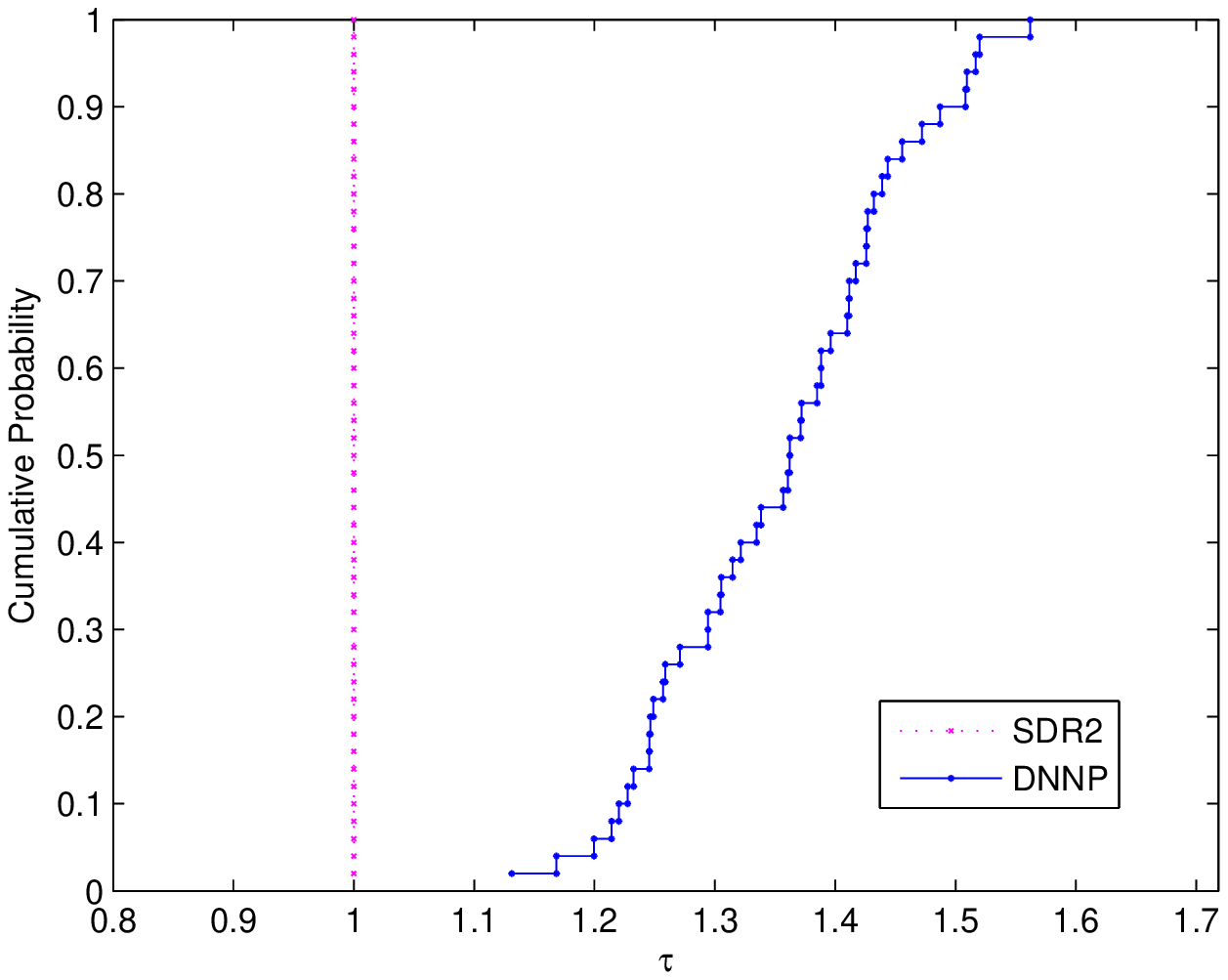}
    \includegraphics[width=67mm]{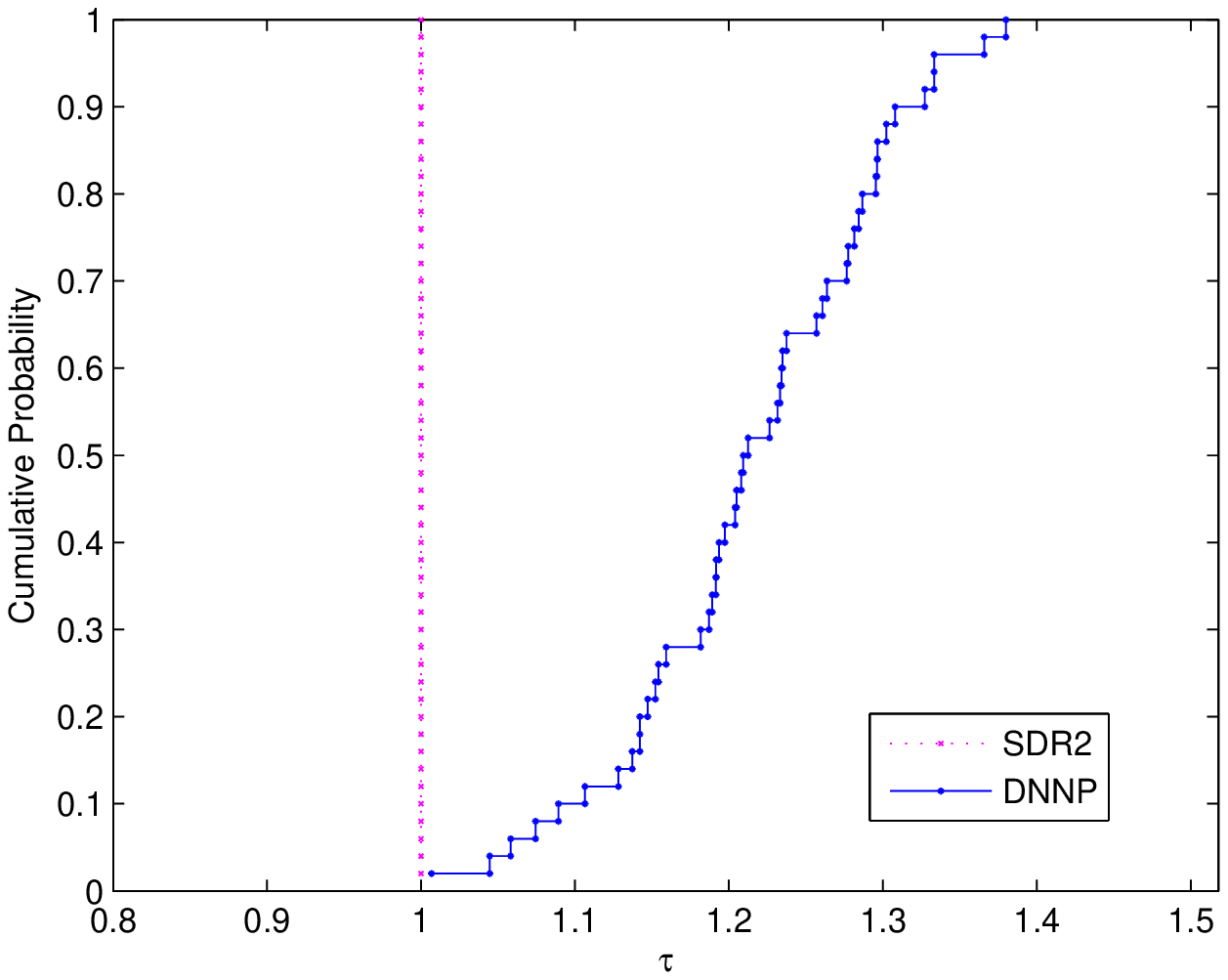}
      \caption{Left figure is based on CPU time of problems RdBQP,
      right figure is based on CPU time of problems RdsBQP.}
     \label{fig_sdr1_sdr2_dnnp_cpu}
 \end{figure}

\section{An application to max-cut problem}\label{sec: application to max-cut
problem}
The max-cut (MC) problem is a kind of important combinatorial
optimization problem on undirected graphs with weights on the edges,
and also is NP-hard \cite{k2010}. Given such a graph, (MC) problem
consists in finding a partition of the set of nodes into two parts
so as to maximize the total weight of edges cut by the partition.

Let $G$ be an $n$-node graph, vertex set $V:=\{1,2,\ldots,n\}$,
$A(G)$ the adjacency matrix of graph $G$, $L$ the Laplacian matrix
associated with the graph, i.e., $L:={\rm{Diag}}(A(G)e)-A(G)$. Let
the vector $u\in\{+1,-1\}^n$ represent any cut in the graph $G$ via
the interpretation that the sets $\{i:u_i=+1\}$ and $\{i:u_i=-1\}$
form a partition of the node set of $G$, we can get the following
formulation for (MC) problem
$$\rm{(MC)}\label{mc}\begin{array}{lll}
&\max &\frac{1}{4}u^TLu\\
&{\rm s.t.} &u\in\{+1,-1\}^n.
 \end{array}
$$

On one hand, by using the standard semidefinite relaxation technique
to (MC) problem, we can get the following problem
$$\rm{\widehat{(SDR)}}\begin{array}{lll}
&\max &\frac{1}{4}L\bullet U\\
&{\rm s.t.} &U_{ii}=1,\ \forall i=1,2,\ldots,n,\\
&&U\succeq 0.
 \end{array}
$$
Goemans and Williamson \cite{gw1995} have provided estimates for the
quality of problem $(\rm\widehat{SDR})$ bound for (MC) problem. By a
randomly rounding a solution to problem $(\rm\widehat{SDR})$, they
propose a $0.878$-approximation algorithm for solving problem (MC)
based on problem $(\rm\widehat{SDR})$, which is known to be the best
approximation ration of polynomial-time algorithm for solving
problem (MC).

On the other hand, according to the technique introduced in Section
\ref{sec:Doubly nonnegative relaxation for problem (BQP)}, problem
(MC) can also be relaxed to the following doubly nonnegative
programming problem
$$\rm{\widehat{(DNNP)}}\begin{array}{lll}
&\max &L\bullet X-x^TLe+\frac{1}{4}e^TLe\\
&{\rm s.t.} &X_{ii}=x_i,\ \forall i=1,2,\ldots,n,\\
&&\left[\begin{array}{ccc}1 & x^T
\\ x & X\end{array}\right]\in S_{1+n}^+\cap(S_{1+n})^+.
 \end{array}
$$

\begin{remark}\label{re: feasible solution}
(i) Note that for two relaxation problems $\rm{\widehat{(SDR)}}$ and
$\rm{\widehat{(DNNP)}}$, the feasible sets are all nonempty. It is
obviously that the identity matrix $E$ is a feasible solution for
problem $\rm{\widehat{(SDR)}}$, and $(x,X)=(0,\mathbf{0})$ feasible
for problem $\rm{\widehat{(DNNP)}}$.

(ii) Compared with problem $\rm{\widehat{(SDR)}}$, problem
$\rm{\widehat{(DNNP)}}$ has not only $n+\frac{n(n+1)}{2}$ new
inequality constraints, but also $n$ variables.
\end{remark}

Thus, according to Theorem \ref{theo: equivalence problem (SDR2) and
problem (DNNP)} and Remark \ref{re: feasible solution}(i), we have
the following theorem.
\begin{theorem}\label{theo: equivalence problem (SDP) and problem (DNNP)}
Problem $\rm\widehat{(SDR)}$ is equivalent to problem
$\rm\widehat{(DNNP)}$.
\end{theorem}
\begin{proof}
On one hand, suppose that $U^*$ is an optimal solution for problem
$\rm\widehat{(SDR)}$, and let $X_{ij}=\frac{1}{4}(U^*_{ij}+1)$ and
$x_i=\frac{1}{2}$, $\forall 1\leq i\leq j\leq n$, i.e.,
\begin{equation}\label{eq:1}
X=\frac{1}{4}(U^*+ee^T),\ x=\frac{1}{2}e,
\end{equation} which imply that
\begin{equation}\label{eq:2}X_{ii}=\frac{1}{4}(U^*_{ii}+1)=x_i=\frac{1}{2}>0,\ \forall
1\leq i\leq n,
\end{equation} since $U_{ii}^*=1$. Then, from $U^*\succeq 0$, it follows that
$$0\leq U^*_{ii}U^*_{jj}-(U^*_{ij})^2=1-(U^*_{ij})^2,$$
i.e., $$-1\leq U^*_{ij}\leq 1,\ \forall 1\leq i<j\leq n,$$ combining
with (\ref{eq:1}), we have
\begin{equation}\label{eq:3}X_{ij}=\frac{1}{4}(U^*_{ij}+1)\geq 0,\ 1\leq i<j\leq
n.\end{equation} Moreover, from (\ref{eq:1}), it follows that
\begin{equation}\label{eq:4}X-xx^T=\frac{1}{4}(U^*+ee^T)-\frac{1}{4}ee^T=\frac{1}{4}U^*\succeq 0,\end{equation}
it followed by $U^*\succeq 0$.
Combining (\ref{eq:2}), (\ref{eq:3}) and (\ref{eq:4}) as well as
Theorem \ref{theo: Schur complement}(ii), it holds that $(x,X)$ is a
feasible solution for problem $\rm\widehat{(DNNP)}$. Again from
(\ref{eq:1}), we have $L\bullet
X-x^TLe+\frac{1}{4}e^TLe=\frac{1}{4}L\bullet
U^*=\rm{Opt}\widehat{(SDR)}$, which further imply that
$\rm{Opt}\widehat{(DNNP)}\geq\rm{Opt}\widehat{(SDR)}$.

On the other hand, suppose that $(x^*,X^*)$ is an optimal solution
for problem $\rm\widehat{(DNNP)}$, and let
\begin{equation}\label{eq:5}U=4X^*-2x^*e^T-2e(x^*)^T+ee^T,\end{equation} which imply that
\begin{equation}\label{eq:6}U_{ii}=4X_{ii}^*-2x^*_i-2x^*_i+1=1,\ \forall 1\leq i\leq n,\end{equation} since
$X_{ii}^*=x_i^*,\ \forall 1\leq i\leq n$.
From (\ref{eq:5}), it follows that
\begin{equation}\label{eq:7}\begin{array}{ll}U&=4X^*-2e(x^*)^T-2x^*e^T+ee^T\\
&=4(X^*-x^*(x^*)^T)+(2x^*-e)(2x^*-e)^T\succeq
0.\end{array}\end{equation} From (\ref{eq:6}) and (\ref{eq:7}), we
can conclude that $U$ defined by (\ref{eq:5}) is a feasible solution
for problem $\rm\widehat{(SDP)}$. Furthermore, again from
(\ref{eq:5}), we have
$$\frac{1}{4}L\bullet U=L\bullet
X^*-(x^*)^TLe+\frac{1}{4}e^TLe=\rm{Opt}\widehat{(DNNP)},$$ which
imply that $\rm{Opt}\widehat{(SDR)}\geq\rm{Opt}\widehat{(DNNP)}$.
The proof is completed.
\end{proof}

By Theorem \ref{theo: equivalence problem (SDP) and problem (DNNP)},
we can obtain doubly nonnegative relaxation for problem (MC) exactly
equal to the standard semidefinite relaxation, i.e. problem
$\rm{\widehat{(DNNP)}}$ and problem $\rm{\widehat{(SDR)}}$ are
equivalent according to Definition \ref{def: equivalence}, without
the boundedness assumption of two feasible sets of two problems.

%
%
\section{Conclusions}\label{sec: conclusions}
In this paper, a class of nonconvex binary quadratic programming
problem is considered, which is NP-hard in general. In order to
solve this problem efficiently by some popular packages for solving
convex programs, two convex representation methods are proposed. One
of the methods is semidefinite relaxation, by the structure of the
binary constraints of original problem, which results in a new
tighter semidefinite relaxation problem (SDR2). The other method is
doubly nonnegative relaxation. The original problem is equivalently
transformed into a convex problem (CPP), which is also NP-hard in
general. Then, by virtue of the features of constraints in this
problem, a computable convex problem (DNNP) is obtained through
doubly nonnegative relaxation. Moreover, the two convex relaxation
problems are equivalent. These results are applied to (MC) problem,
we can conclude that doubly nonnegative relaxation for problem (MC)
is equivalent to the standard semidefinite relaxation for it.
Furthermore, some compared numerical results are reported to show
the performance of two relaxed problems.


\bibliographystyle{model1-num-names}
\bibliography{references_full_name}
\end{document}